\documentclass[12pt]{article}
\usepackage{amssymb,amsmath,latexsym}

\usepackage{amssymb}
\usepackage{amsmath}
 \usepackage{amsthm}
\newtheorem{theorem}{Theorem}
\newtheorem{proposition}[theorem]{Proposition}

\theoremstyle{definition}
\newtheorem{definition}{Definition}
\newtheorem{example}{Example}
\newtheorem{remark}{Remark}

\newtheorem*{conjecture}{Conjecture}

\begin{document}
\title{On exponential stability for linear discrete-time systems in Banach spaces}
\author{Ioan-Lucian Popa*, Traian Ceau\c su*, Mihail Megan$^{*,\ddag}$}
\date{}
\maketitle
\abstract{In this paper we investigate four concepts of
exponential stability for difference equations in Banach spaces.
Characterizations of these concepts are given. They can be
considered as variants for the discrete-time case of the classical
results due to  E.A. Barbashin \cite{barbasin} and R. Datko
\cite{datko}. An illustrative example clarifies the relations
between these concepts.}

\vspace{1cm}
{\it Keywords:} linear discrete-time systems, uniform
exponential stability, (nonuniform) exponential stability,
strongly
exponential stability\\

\let\thefootnote\relax\footnote{$^\ddag$ Academy of Romanian Scientists, Independen\c tei 54,
050094 Bucharest, Romania\\
$^{*,\ddag}$ Department of Mathematics, Faculty of Mathematics and
Computer Science, West University of Timi\c soara, V.  P\^ arvan
Blv. No. 4, 300223-Timi\c soara, Romania}

\section{Introduction}

Let $X$ be a real or complex Banach space, $\mathcal{B}(X)$ be the
Banach algebra of all bounded linear operators from $X$ into
itself, and $X^{*}$ be the Banach space of all continuous linear
functionals on $X$ (the dual space of $X$).

The norms on $X$, in $\mathcal{B}(X)$ and occasionally in the dual
space $X^{*}$ are all denoted by $\parallel . \parallel.$ Let
$\Delta$ be the set of all pairs $(m,n)$ of positive integers
satisfying the inequality $m \geq n.$ We also denote by $T$ the
set of all triplets $(m,n,p)$ of positive integers with $(m,n)$
and $(n,p)\in\Delta .$

We consider the linear discrete-time system
\begin{equation*}\tag{$\mathbf{A}$}\label{A}
x_{n+1}=A{(n)}x_{n},
\end{equation*}
where $A:\mathbb{N}\rightarrow \mathcal{B}(X)$ is a given
$\mathcal{B}(X)-$valued sequence.

For $(m,n)\in\Delta$ we define
\begin{equation}\label{eqAmn}
A_{m}^{n}=\left\{\begin{array}{ll}
A{(m)}\cdot \ldots\cdot  A{(n+1)},\;\; \;\; m \geq n+1\\
\qquad\quad I\qquad\qquad\quad,\;\;\;\; m=n.
\end{array}\right.
\end{equation}
(where $I$ is the identity operator on $X$).

It is obvious that $A_{m}^{n}A_{n}^{p}=A_{m}^{p}\;\; \text{for
all}\;\; (m,n,p)\in T.$

In the theory of difference equations both in finite and infinite
dimensional spaces, the concepts of exponential stability play a
central role in the study of the asymptotical behaviors of
solutions of discrete-time systems. In this sense we recall the
classical monographs due to R.P. Agarwal \cite{agarwal}, S. Elaydi
\cite{Elaydi1}, M.I. Gil \cite{mil}, L. Lakshmikantham and D.
Trigiante \cite{lasmikantham}, where the stability properties of
discrete-time systems are studied.

In this paper we consider four concepts of exponential stability
for linear discrete-time systems: uniform exponential stability,
nonuniform exponential stability, strong exponential stability and
exponential stability.

Our main objective is to obtain appropriate versions of the
well-known stability theorems due (in the continuous case) to R.
Datko \cite{datko} and E.A. Barbashin \cite{barbasin}. An
illustrative example clarifies the implications between these
exponential stability concepts.

We remark that in comparison with the classical notion of uniform
exponential stability, the concepts of strong exponential
stability and exponential stability are much weaker behaviors.

A principal motivation for weakening the assumption of uniform
exponential behavior is that from the point of view of ergodic
theory, almost all variational equations in  finite dimensional
spaces have a nonuniform exponential behavior.
%
%
\section{ Uniform exponential stability}
In this section we consider the well-known concept of uniform
exponential stability of a linear discrete-time system (\ref{A})
given by:
\begin{definition}\label{defues}
The linear discrete-time system (\ref{A}) is said to be {\it
uniformly exponentially stable} (and denoted as u.e.s) if there
are some
 constants $N\geq 1$ and $\alpha>0$ such that
 \begin{equation}\label{eques}
 \parallel A_{m}^{p}x\parallel \leq N e^{-\alpha(m-n)}\parallel
 A_{n}^{p}x\parallel,\;\;\text{for all}\;\;(m,n,p,x)\in T\times X.
 \end{equation}
\end{definition}
\begin{remark}\label{remues}
The linear discrete-time system (\ref{A})  is uniformly
exponentially stable if and only if there are some  constants
$N\geq 1$ and $\alpha>0$ such that
 \begin{equation}\label{eqremues}
 \parallel A_{m}^{n}x\parallel \leq N e^{-\alpha(m-n)}\parallel
 x\parallel,\;\;\text{for all}\;\;(m,n,x)\in \Delta\times X.
 \end{equation}
\end{remark}
A preliminary result for uniform exponential stability is given
by:
\begin{proposition}\label{prop1}
For every linear discrete-time system (\ref{A}) the following
statements are equivalent:
\begin{description}

\item[{\it{i)}}] (\ref{A}) is uniformly exponentially stable;

\item[{\it{ii)}}] there exist two constants $N \geq 1$ and $a\in
(0,1)$ such that
\begin{equation}
 \parallel A_{m}^{n}x\parallel \leq N a^{m-n}\parallel
 A_{n}^{p}x\parallel,\;\;\text{for all}\;\;(m,n,p,x)\in T\times X;
 \end{equation}

\item[{\it{iii)}}] there exist a constant $N \geq 1$ and a
sequence of positive  real numbers $(a_n)$ with $a_{n}\rightarrow
0$ such that
\begin{equation}\label{p1eqiii}
 \parallel A_{m}^{n}x\parallel \leq N a_{m-n}\parallel A_{n}^{p}x\parallel,\;\;\text{for all}\;\;(m,n,p,x)\in T\times X;
 \end{equation}

\item[{\it{iv)}}] there exist a constant $N \geq 1$ and a
 sequence of positive real numbers $(a_{n})$ with
 $a_{n}\rightarrow 0$ such that
 \begin{equation}
\parallel A_{m}^{n}x\parallel \leq N a_{m-n}\parallel x\parallel,\;\;\text{for all}\;\;(m,n,x)\in \Delta\times X.
 \end{equation}
\end{description}
\end{proposition}
\begin{proof}
The implications $(i)\Rightarrow (ii)\Rightarrow (iii) \Rightarrow
(iv)$ are obvious.

$(iv)\Rightarrow (i)$  We observe that from $a_{n}\rightarrow 0$
it follows that there  exists a positive integer $k$ with $Na_k
<1.$ Then for all $(m,n)\in\Delta$ there exist two positive
integers $r$ and $s$ with $0\leq r <k$ such that $m-n=ks+r.$ If we
define $\alpha = -\dfrac{\ln (Na_k)}{k}
>0$ then:

$(a)$ For the case $s=0$ we obtain
\begin{align*}
\parallel A_{m}^{n}x\parallel &\leq N a_{r} \parallel
x\parallel\leq N e^{\alpha r}e^{-\alpha r}\parallel x\parallel\\
&\leq N e^{-\ln (N a_{k})} e ^{-\alpha (m-n)}\parallel x\parallel
= M e^{-\alpha (m-n)}\parallel x\parallel
\end{align*}
for all $(m,n,x)\in\Delta\times X.$

$(b)$ If $s \neq 0$ we have
\begin{align*}
\parallel A_{m}^{n}x\parallel &\leq N a_{r} \parallel A_{n+sk}^{n}x\parallel \leq \ldots \leq N (N a_{k})^{s} \parallel x\parallel\leq\\
&\leq N e^{\alpha k} e^{-\alpha (m-n)}\parallel x\parallel
=Me^{-\alpha (m-n)} \parallel x\parallel.
\end{align*}
for all $(m,n,x)\in\Delta\times X.$ Thus, (\ref{A}) is u.e.s.,
which ends the proof.
\end{proof}
Other characterizations of the uniform exponential stability
property are given by:
\begin{theorem}\label{theorem1}
For every linear discrete-time system (\ref{A}) the following
assertions are equivalent:
\begin{description}

\item[{\it{i)}}] (\ref{A}) is uniformly exponentially stable;

\item[{\it{ii)}}] there are some constants $D \geq 1$ and $d >0$
such that
\begin{equation}
\sum\limits_{m=n}^{\infty} e^{d(m-n)}\parallel A_{m}^{p}x\parallel
\leq D \parallel A_{n}^{p}x\parallel,\;\;\text{for
all}\;\;(n,p,x)\in \Delta\times X.
\end{equation}

\item[{\it{iii)}}] there exists a constant $D \geq 1$ such that
\begin{equation}\label{eqd2iii}
\sum\limits_{m=n}^{\infty}\parallel A_{m}^{p} x\parallel \leq
D\parallel A_{n}^{p}x\parallel,\;\;\text{for all}\;\;(n,p,x)\in
\Delta\times X.
\end{equation}

\item[{\it{iv)}}] there exists a constant $D \geq 1$ such that
\begin{equation}
\sum\limits_{m=n}^{\infty}\parallel A_{m}^{n} x\parallel \leq
D\parallel x\parallel,\;\;\text{for all}\;\;(n,x)\in
\mathbb{N}\times X.
\end{equation}
\end{description}
\end{theorem}
\begin{proof}

The implications $(i)\Rightarrow (ii)\Rightarrow (iii)\Rightarrow
(iv)$ are immediate.

The implication $(iv)\Rightarrow (i)$ has been proved in
\cite{przyluski}.
\end{proof}
\begin{remark}
The preceding theorem can be considered as a discrete-time variant
of the well-known theorem due to Datko \cite{datko}. Similar
results are obtained in \cite{przyluski}, \cite{predamegan},
\cite{agarwal}.
\end{remark}

\begin{theorem}\label{theorem2}
The following statements are equivalent:
\begin{description}

\item[{\it{i)}}]  (\ref{A}) is uniformly exponentially stable;

\item[{\it{ii)}}] there are some constants $B \geq 1$ and $b >0$
such that
\begin{equation}
\sum\limits_{k=n}^{m} e^{b(m-k)}\parallel
(A_{m}^{k})^{*}x^{*}\parallel \leq B\parallel
x^{*}\parallel,\;\;\text{for all}\;\;(m,x^{*})\in \mathbb{N}\times
X^{*};
\end{equation}

\item[{\it{iii)}}] there exists a constant $B \geq 1$ such that
\begin{equation}\label{eqd23}
\sum\limits_{k=0}^{m}\parallel (A_{m}^{k})^{*}x^{*}\parallel \leq
B\parallel x^{*}\parallel,\;\;\text{for all}\;\;(m,x^{*})\in
\mathbb{N}\times X^{*};
\end{equation}
\end{description}
\end{theorem}
\begin{proof}

$(i)\Rightarrow (ii)$ Using Definition \ref{defues} we have that
$$\sum\limits_{k=0}^{m} e^{b(m-k)}\parallel (A_{m}^{k})^{*}x^{*}\parallel \leq \sum\limits_{k=0}^{m}e^{b(m-k)} N e^{-\alpha(m-k)}\parallel x^{*}\parallel \leq \frac{N
e^{\alpha}}{e^{\alpha}-e^{b}}\parallel x^{*}\parallel$$ for any
$b\in (0,\alpha)$ and all $(m,x^{*})\in \mathbb{N}\times\ X^{*}.$
\newline
$(ii)\Rightarrow (iii)$ This is obvious.
\newline
$(iii)\Rightarrow (i)$ Let $(m,n,x^{*})\in\Delta\times X^{*}.$ By
our hypothesis for $k\in \{n,n+1,\ldots , m\}$ and all $y^{*}\in
X^{*}$  we have that $\parallel (A_{k}^{n})^{*}y^{*}\parallel \leq
B
\parallel y^{*}\parallel.$ If we consider
$y^{*}=(A_{m}^{k})^{*}x^{*}$ we obtain $\parallel
(A_{m}^{n})^{*}x^{*}\parallel \leq \parallel
(A_{k}^{n})^{*}(A_{m}^{k})^{*}x^{*}\parallel\leq B \parallel
(A_{m}^{k})^{*}x^{*}\parallel.$
\begin{align*}
\text{Hence,}\;\;(m-n+1)\parallel (A_{m}^{n})^{*}x^{*}\parallel &=\sum\limits_{k=n}^{m}\parallel (A_{m}^{n})^{*}x^{*}\parallel\leq B\sum\limits_{k=n}^{m}\parallel (A_{m}^{k})^{*}x^{*}\parallel\\
& \leq B\sum\limits_{k=0}^{m}\parallel
(A_{m}^{k})^{*}x^{*}\parallel\leq  B^{2}\parallel x^{*}\parallel.
\end{align*}

Thus,  we can conclude that equation (\ref{A}) is u.e.s.
\end{proof}
\begin{remark}
The previous theorem  is an analog for the linear discrete-time
systems of a well-known theorem proved in the continuous case by
E.A. Barbashin (\cite{barbasin}, Theorem 5.1, p. 169).
\end{remark}
%
%
\section{Nonuniform exponential stability}
%
\begin{definition}
The linear discrete-time system (\ref{A}) is said to be {\it
nonuniformly exponentially stable} (and denoted as n.e.s.) if
there exists a constant $\alpha >0$ and a nondecreasing function
$N:\mathbb{R}_{+}\rightarrow [1,\infty)$ such that
\begin{equation}\label{defnes}
\parallel A_{m}^{p}x\parallel \leq N(n) e^{-\alpha
(m-n)}\parallel A_{n}^{p}x\parallel,\;\;\text{for
all}\;\;(m,n,p,x)\in T\times X.
\end{equation}
\end{definition}
\begin{remark}\label{remarcanes}
The linear discrete-time system (\ref{A}) is n.e.s. if and only if
there exist a constant $\alpha >0$ and a nondecreasing function
$N:\mathbb{R}_{+}\rightarrow [1,\infty)$ such that
\begin{equation}\label{remnes}
\parallel A_{m}^{n}x\parallel \leq N(n) e^{-\alpha
(m-n)}\parallel x\parallel,\;\;\text{for all}\;\;(m,n,x)\in
\Delta\times X.
\end{equation}
\end{remark}
It is obvious that if (\ref{A})  is u.e.s. then it is n.e.s. The
following example shows that the converse implication is not
valid.
\begin{example}\label{exemplul1}
Let (\ref{A}) be the linear discrete-time system given by
\begin{equation}\label{a_n}
 A(n)=c
a_{n}I,\;\;\text{where}\;\; a_{n} = \left\{
  \begin{array}{l l}
    e^{-n} & \quad \text{if $n=2k$}\\
    e^{n+1} & \quad \text{if $n=2k+1$}\\
  \end{array} \right.
\end{equation}
with $c\in\left(0,\dfrac{1}{e}\right).$ Let
$(m,n,x)\in\Delta\times X.$ According to (\ref{eqAmn}) we have
that
\[
  A_{m}^{n}x = \left\{
  \begin{array}{l l}
    c^{m-n}a_{mn}x & \quad m>n\\
    x & \quad m=n\\
  \end{array} \right. ,
\]
 where
\[
  a_{mn} = \left\{
  \begin{array}{l l}
    1 & \quad \text{if $m=2q$ and $n=2p$}\\
    e^{-n-1} & \quad \text{if $m=2q$ and $n=2p+1$}\\
    e^{m+1} & \quad \text{if $m=2q+1$ and $n=2p$}\\
    e^{m-n} & \quad \text{if $m=2q+1$ and $n=2p+1$}\\
  \end{array} \right.
\]
Firstly, we observe that if we suppose that equation (\ref{A}) is
u.e.s. then there exist some constants $N \geq 1$ and $\alpha >0$
such that
$$\parallel A_{m}^{n}x\parallel \leq N e^{-\alpha (m-n)} \parallel
x\parallel,\;\;\text{for all}\;\;(m,n,x)\in\Delta\times X.$$

In particular, for $m=2q+1$ and $n=2q$ it results that $e^{\alpha
+2q+2}\leq N,$ for all positive integers $q$, which is a
contradiction. Hence, (\ref{A}) is not u.e.s. for every $c>0.$

If $c\in\left(0,\dfrac{1}{e}\right)$ then for $\alpha=\ln
\dfrac{1}{ce} >0$ and $N(n)=e^{n+1}$ we have that
\begin{equation}\label{steluta}
\parallel
A_{m}^{n}x\parallel \leq N(n)e^{-\alpha (m-n)} \parallel
x\parallel,\;\;\text{for all}\;\;(m,n,x)\in\Delta\times X,
\end{equation}
which shows that (\ref{A}) is n.e.s.
\end{example}
\begin{theorem}\label{theorem1nes}
The linear discrete-time system (\ref{A}) is nonuniformly
exponentially stable if and only if  there exists a  constant $d
>0$ and a nondecreasing function $N:\mathbb{R}_{+}\rightarrow
[1,\infty)$ such that
\begin{equation}\label{eq1th1nes}
\sum\limits_{m=n}^{\infty} e^{d(m-n)}\parallel A_{m}^{p}x\parallel
\leq  N(n)\parallel A_{n}^{p}x\parallel,\;\;\text{for
all}\;\;(m,n,p,x)\in T\times X.
\end{equation}
\end{theorem}
\begin{proof}
{\it Necessity.} It is a simple verification.
\newline
{\it Sufficiency.} The inequality (\ref{eq1th1nes}) implies that
$e^{d(m-n)}\parallel A_{m}^{n}x\parallel\leq N(n)\parallel
x\parallel,$ for all $(m,n,x)\in\Delta\times X.$ By Remark
\ref{remarcanes} it results that (\ref{A}) is n.e.s.
\end{proof}
\begin{theorem}\label{theorem2nes}
 If there is a constant $b >0$ and a nondecreasing function
$N:\mathbb{R}_{+}\rightarrow [1,\infty)$ such that
\begin{equation}\label{eq1th2nes}
\sum\limits_{k=n}^{m} e^{b(m-k)}\parallel
(A_{m}^{k})^{*}x^{*}\parallel \leq N(n)\parallel
x^{*}\parallel,\;\;\text{for all}\;\;(m,n,x^{*})\in \Delta\times
X^{*},
\end{equation}
then the linear discrete-time system (\ref{A}) is nonuniformly
exponentially stable.
\end{theorem}
\begin{proof}
By (\ref{eq1th2nes}) we have that $\parallel A_{m}^{n}\parallel
\leq N(n) e^{-b(m-n)},$ for all $(m,n)\in\Delta,$ which implies
that (\ref{A}) is n.e.s.
\end{proof}
%
%
%
\section{Exponential Stability}
In this section we study a particular concept of nonuniform
exponential stability, studied by L. Barreira and C. Valls (see
for example \cite{bareira},\cite{bareira1}).
\begin{definition}\label{defes}
The linear discrete-time system (\ref{A}) is said to be {\it
exponentially stable} (and denoted as e.s) if there are some
 constants $N\geq 1,$ $\alpha>0$ and $\beta \geq 0$ such that
 \begin{equation}\label{es}
 \parallel A_{m}^{p}x\parallel \leq N e^{-\alpha(m-n)}e^{\beta n}\parallel
 A_{n}^{p}x\parallel,\;\;\text{for all}\;\;(m,n,p,x)\in T\times X.
 \end{equation}
\end{definition}
\begin{remark}\label{remark4}
The linear discrete-time system (\ref{A}) is  exponentially stable
if and only if there are some
 constants $N\geq 1,$ $\alpha>0$ and $\beta \geq 0$ such that
 \begin{equation}\label{rem4ues}
 \parallel A_{m}^{n}x\parallel \leq N e^{-\alpha(m-n)}e^{\beta n}\parallel
 x\parallel,\;\;\text{for all}\;\;(m,n,x)\in \Delta\times X.
 \end{equation}
\end{remark}
\begin{proposition}\label{prop2}
The following statements are equivalent:
\begin{description}

\item[{\it{i)}}] (\ref{A}) is exponentially stable;

\item[{\it{ii)}}] There exist some constants $N \geq 1,$ $\nu >0$
and $\beta \in [0,\nu)$ such that
\begin{equation}\label{eq2prop2}
 \parallel A_{m}^{p}x\parallel \leq N e^{-\nu (m-n)}e^{\beta m}\parallel
 A_{n}^{p}x\parallel,\;\;\text{for all}\;\;(m,n,p,x)\in T\times X.
 \end{equation}

 \item[{\it{iii)}}]There exist some constants $N
\geq 1,$ $\nu >0$ and $\delta >0$ with $\delta \leq \nu$  such
that
\begin{equation}\label{prop2eqiii}
 \parallel A_{m}^{p}x\parallel \leq N e^{-\delta m}e^{\nu n}\parallel
 A_{n}^{p}x\parallel,\;\;\text{for all}\;\;(m,n,p,x)\in T\times X.
 \end{equation}
\end{description}
\end{proposition}
\begin{proof}
$(i)\Rightarrow (ii)$ It is a simple verification for $\beta \geq
0$ and $\nu =\alpha +\beta,$ with $\alpha$ and $\beta$ given by
Definition \ref{defes}.
\newline
$(ii)\Rightarrow (iii)$ We have that $\parallel
A_{m}^{p}x\parallel \leq N e^{\nu n}e^{- (\nu - \beta)m}\parallel
A_{n}^{p}x\parallel.$ Hence, for $\nu >0,$ and $\delta =\nu
-\beta$ we obtain relation (\ref{prop2eqiii}).
\newline
$(iii)\Rightarrow (i)$ Using (\ref{prop2eqiii}) we obtain that
$\parallel A_{m}^{p}x\parallel \leq N e^{-\delta (m-n)}e^{(\nu
-\delta)n}\parallel A_{n}^{p}x\parallel.$
\newline
For $\alpha = \delta$ and $\beta =\nu -\delta$ we obtain that
equation (\ref{A}) is e.s.
\end{proof}
\begin{remark}
It is obvious that  $u.e.s.\Rightarrow e.s.$ The following example
shows that the converse implication is not valid.
\end{remark}
\begin{example}\label{exemplues}
Let (\ref{A}) be the discrete-time system given by  (\ref{a_n}),
with $c>0.$ From (\ref{steluta}) it results that if $c\in \left(
0,\dfrac{1}{e}\right)$ then (\ref{A}) is e.s.

The converse implication is also true, because if we suppose that
(\ref{A}) is e.s. and $c\geq \dfrac{1}{e}$ then for $m=2q+1,$
$n=1$ with $q\in\mathbb{N}$ we obtain $N\geq (e^{\alpha
+1}c)^{2q}e^{-\beta}$ which is  false for $q\rightarrow\infty.$ In
conclusion (\ref{A}) is e.s. if and only if $c\in \left(
0,\dfrac{1}{e}\right).$
\end{example}
\begin{theorem}\label{teorema6}
The linear discrete-time system (\ref{A})  is exponentially stable
if and only if there exist some constants $D \geq 1,$ $d> 0$ and
$c \geq 0$ such that
\begin{equation}\label{t6eqes}
\sum\limits_{m=n}^{\infty} e^{d(m-n)}\parallel  A_{m}^{n}
x\parallel \leq D e^{c n} \parallel x\parallel,\;\;\text{for
all}\;\;(n,x)\in \mathbb{N}\times X.
\end{equation}
\end{theorem}
\begin{proof}
{\it Necessity.} It is a simple verification  for $c=\beta ,$
$d\in (0,\alpha)$ and $D=1+
\dfrac{Ne^{\alpha}}{e^{\alpha}-e^{d}},$ with $N,$ $\alpha$ and
$\beta$ offered by Remark \ref{remark4}.
\newline
{\it Sufficiency.} The inequality (\ref{t6eqes}) implies
 $\parallel A_{m}^{n}x\parallel \leq D e^{cn}
 e^{-d(m-n)}\parallel x\parallel,$ for all
 $(m,n,x)\in\Delta\times X,$
 which shows that (\ref{A}) is
 e.s.
\end{proof}
\begin{theorem}\label{teorema8}
The linear discrete-time system (\ref{A})  is exponentially stable
if and only if there exist some constants $B \geq 1,$ $b>0$ and $c
\in [0,b)$ such that
\begin{equation}\label{t8eqtbes}
\sum\limits_{k=0}^{m} e^{b(m-k)}\parallel
(A_{m}^{k})^{*}x^{*}\parallel \leq B e^{c m}\parallel
x^{*}\parallel,\;\;\text{for all}\;\;(m,x^{*})\in \mathbb{N}\times
X^{*}.
\end{equation}
\end{theorem}
\begin{proof}
{\it Necessity.} It is a simple verification for  $c=\beta ,$
$b\in (\beta ,\alpha+\beta )$ and $B=1+\dfrac{Ne^{\alpha +\beta
}}{e^{\alpha + \beta }-e^{b}},$ with $N,$ $\alpha$ and $\beta$
given by Remark \ref{remark4}.
\newline
{\it Sufficiency.} Since $0\leq c <b,$ using inequality
(\ref{t8eqtbes}) for all $(m,n)\in\Delta$ we have that
$e^{b(m-n)}\parallel A_{m}^{n}\parallel =e^{b(m-n)}\parallel
(A_{m}^{n})^{*}\parallel \leq B e^{cm},$ which proves that
(\ref{A}) is e.s.
\end{proof}
%
%
\section{Strong exponential stability}
A particular concept of exponential stability is defined by:
\begin{definition}\label{defses}
The linear discrete-time system (\ref{A}) is said to be {\it
strongly exponentially stable} (and denoted as s.e.s) if there are
some
 constants $N\geq 1,$ $\alpha>0$ and $\beta \in [0,\alpha )$ such that
 \begin{equation}\label{es ses}
 \parallel A_{m}^{p}x\parallel \leq N e^{-\alpha(m-n)}e^{\beta n}\parallel
 A_{n}^{p}x\parallel,\;\;\text{for all}\;\;(m,n,p,x)\in T\times X.
 \end{equation}
\end{definition}
\begin{remark}\label{remark5}
The linear discrete-time system (\ref{A}) is strongly
exponentially stable  if and only if there are some
 constants $N\geq 1,$ $\alpha>0$ and $\beta \in [0,\alpha)$ such that
 \begin{equation}\label{ses}
 \parallel A_{m}^{n}x\parallel \leq N e^{-\alpha(m-n)}e^{\beta n}\parallel
 x\parallel,\;\;\text{for all}\;\;(m,n,x)\in \Delta\times X.
 \end{equation}
\end{remark}
\begin{proposition}\label{prop3}
The linear discrete-time system (\ref{A}) is strongly
exponentially stable if and only if there exist some constants $N
\geq 1,$ $\alpha >0$ and $\nu > 0$ with $\alpha \leq \nu <
2\alpha$ such that
\begin{equation}\label{eqprop3}
 \parallel A_{m}^{p}x\parallel \leq N e^{-\alpha m}e^{\nu n}\parallel
 A_{n}^{p}x\parallel,\;\;\text{for all}\;\;(m,n,p,x)\in T\times X.
 \end{equation}
\end{proposition}
\begin{proof}
{\it Necessity.} Using our hypothesis and Definition \ref{defses}
 we have that
$$\parallel A_{m}^{n}x\parallel \leq N e^{-\alpha (m-n)}e^{\beta
n}\parallel A_{n}^{p}\parallel =N e^{-\alpha m}e^{(\beta +\alpha)
n}\parallel A_{n}^{p}x\parallel$$ and for $\alpha >0$ and $\nu
=\beta +\alpha$ we obtain relation (\ref{eqprop3}).
\newline
{\it Sufficiency.} We have that $\parallel A_{m}^{p}x\parallel
\leq N e^{(\nu -\alpha)n}e^{-\alpha (m-n)}\parallel
A_{n}^{p}x\parallel.$ For $\alpha >0$ and $\beta =\nu -\alpha$ we
obtain that (\ref{A}) is s.e.s.
\end{proof}
\begin{remark}
It is obvious that $s.e.s. \Rightarrow e.s.$ The following example
shows that the concept of strong exponential stability is a
distinct concept of exponential stability.
\end{remark}
\begin{example}\label{exampleses}
Let (\ref{A}) be the linear discrete-time system given by
(\ref{a_n}).

If $c\in \left[ \dfrac{1}{e^{2}},\dfrac{1}{e}\right)$ then from
the considerations given in Example \ref{exemplues} it results
that (\ref{A}) is e.s. If we suppose that (\ref{A}) is s.e.s. then
there are $N \geq 1,$ $\alpha >0$ and $\beta\in [0,\alpha)$ such
that $\parallel A_{m}^{n}x\parallel \leq N e^{\beta n}e^{-\alpha
(m-n)}\parallel x\parallel,$ for all $(m,n,x)\in\Delta\times X.$
Then for $m=2q+1$ and $n=1$ we obtain $N \geq (e^{\alpha
+1}c)^{2q}e^{-\beta}$ which is impossible. Hence, for $c\geq
\dfrac{1}{e^{2}},$ (\ref{A}) is not s.e.s., which shows that if
(\ref{A}) is s.e.s. then $c\in\left(0,\dfrac{1}{e^{2}}\right).$
The converse implication is also valid. Indeed, for
$c\in\left(0,\dfrac{1}{e^{2}}\right),$ $\alpha=-\ln (ce),$
$\beta=1$ and $N=e$ we have that $\parallel A_{m}^{n}x\parallel
\leq N e^{\beta n}e^{-\alpha (m-n)}\parallel x\parallel,$ for all
$(m,n,x)\in\Delta\times X,$ with $\alpha >\beta \geq 1.$ Thus,
(\ref{A}) is s.e.s. In conclusion, (\ref{A}) is s.e.s. if and only
if $c\in\left(0,\dfrac{1}{e^{2}}\right).$
\end{example}

\begin{theorem}\label{teorema7}
The linear discrete-time system (\ref{A})  is strongly
exponentially stable if and only if there exist some constants $D
\geq 1,$ $d> 0$ and $c \geq 0$ with $0\leq c < d$ such that
\begin{equation}\label{t7etes}
\sum\limits_{m=n}^{\infty} e^{d(m-n)}\parallel  A_{m}^{n}
x\parallel \leq D e^{c n} \parallel x\parallel,\;\;\text{for
all}\;\;(n,x)\in \mathbb{N}\times X.
\end{equation}
\end{theorem}
\begin{proof}
It results from the Definition \ref{defses} and the proof of
Theorem \ref{teorema6}.
\end{proof}
\begin{theorem}\label{teorema9}
The linear discrete-time system (\ref{A})  is strongly
exponentially stable if and only if there exist some constants $B
\geq 1,$ $b>0$ and $c \geq 0$ with $0\leq 2c<b$ such that
\begin{equation}\label{eqtbes}
\sum\limits_{k=0}^{m} e^{b(m-k)}\parallel
(A_{m}^{k})^{*}x^{*}\parallel \leq B e^{c m}\parallel
x^{*}\parallel,\;\;\text{for all}\;\;(m,x^{*})\in \mathbb{N}\times
X^{*}.
\end{equation}
\end{theorem}
\begin{proof}
It results from  Definition \ref{defses} and the proof of Theorem
\ref{teorema8}.
\end{proof}
%
%
\section{Conclusions}
The paper considers four concepts of exponential stability for
linear discrete-time systems in Banach spaces. We establish
relations between these concepts. It is obvious that we have
$u.e.s. \Rightarrow s.e.s.\Rightarrow e.s. \Rightarrow n.e.s.$ and
to emphasize the straightforwardness of the implication we have
considered the linear discrete-time system (\ref{A}) given by
$A(n)= c a_{n}I,$ where $c\geq 0$ and

\[
  a_{n} = \left\{
  \begin{array}{l l}
    e^{-n} & \quad \text{if $n=2k$}\\
    e^{n+1} & \quad \text{if $n=2k+1$}\\
  \end{array} \right.
\]
Examples \ref{exemplul1}-\ref{exampleses} show that:

$i)$ (\ref{A}) is u.e.s. if and only if $c=0;$

$ii)$ (\ref{A}) is n.e.s. if and only if
$c\in\left(0,\dfrac{1}{e}\right);$

$iii)$ (\ref{A}) is e.s. if and only if
$c\in\left(0,\dfrac{1}{e}\right);$

$iv)$ (\ref{A}) is s.e.s. if and only if
$c\in\left(0,\dfrac{1}{e^{2}}\right).$

The Theorems \ref{theorem1}, \ref{theorem1nes}, \ref{teorema6} and
\ref{teorema7} give characterizations for the exponential
stability concepts studied in this paper. They can be considered
as variants for the discrete-time case of a well-known result due
to Datko (\cite{datko}) in the continuous case.

The characterizations given by Theorems \ref{theorem2},
\ref{theorem2nes}, \ref{teorema8} and \ref{teorema9} can be
considered as variants for the discrete-time case of a result due
to E.A. Barbashin (\cite{barbasin}, Theorem 5.1, p. 169) in the
continuous case.
\section{Open problems}
Our goal here is to present an interesting open problem.A
characterization of Barbashin type for u.e.s. is given in
\cite{predamegan}. This result implies:
\begin{theorem}
For every linear discrete-time system (\ref{A}) the following
statements are equivalent:
\begin{description}

\item[{\it{i)}}] (\ref{A}) is uniformly exponentially stable;

\item[{\it{ii)}}] there are some constants $B \geq 1$ and $b >0$
such that
\begin{equation}
\sum\limits_{k=0}^{m} e^{b(m-k)}\parallel A_{m}^{k}\parallel \leq
B,\;\;\text{for all}\;\;m\in \mathbb{N};
\end{equation}

\item[{\it{iii)}}] there exists a constant $B \geq 1$ such that
\begin{equation}\label{eqd2345}
\sum\limits_{k=0}^{m}\parallel A_{m}^{k}\parallel \leq
B,\;\;\text{for all}\;\;m\in \mathbb{N}.
\end{equation}
\end{description}
\end{theorem}
Similar results hold for other concepts of exponential stability.
We present as an open problem the proof of the implication
$(iii)\Rightarrow (i)$ for the following:
\begin{conjecture}
For every linear discrete-time system (\ref{A}) the following
statements are equivalent:
\begin{description}

\item[{\it{i)}}] (\ref{A}) is uniformly exponentially stable;

\item[{\it{ii)}}] there are some constants $B \geq 1$ and $b >0$
such that
\begin{equation}
\sum\limits_{k=0}^{m} e^{b(m-k)}\parallel A_{m}^{k}x\parallel \leq
B\parallel x\parallel,\;\;\text{for all}\;\;(m,x)\in
\mathbb{N}\times X;
\end{equation}

\item[{\it{iii)}}] there exists a constant $B \geq 1$ such that
\begin{equation}\label{eqd234}
\sum\limits_{k=0}^{m}\parallel A_{m}^{k}x\parallel \leq B\parallel
x\parallel,\;\;\text{for all}\;\;(m,x)\in \mathbb{N}\times X.
\end{equation}
\end{description}
\end{conjecture}
The implications $(i)\Rightarrow(ii)\Rightarrow(iii)$ are
immediate. The problem is the proof of the implication
$(iii)\Rightarrow(i)$. Of course, we are interested in obtaining
similar results for n.e.s., e.s. and s.e.s according to the
previous condition.

\end{document}